\documentclass[sn-mathphys,Numbered]{sn-jnl}

\usepackage{hyperref}
\usepackage{rotating,mathrsfs,amssymb,multirow,subfigure,float}

\usepackage{graphicx}%
\usepackage{multirow}%
\usepackage{amsmath,amssymb,amsfonts}%
\usepackage{amsthm}%
\usepackage{mathrsfs}%
\usepackage[title]{appendix}%
\usepackage{xcolor}%
\usepackage{textcomp}%
\usepackage{manyfoot}%
\usepackage{booktabs}%
\usepackage{algorithm}%
\usepackage{algorithmicx}%
\usepackage{algpseudocode}%
\usepackage{listings}%
\usepackage{amssymb,amsmath}

\usepackage{amsthm,mathabx}
\usepackage{amsthm}
\usepackage{graphicx,multirow,amsmath,amsfonts,mathrsfs,amssymb}
\usepackage{amsthm}
\usepackage{bm,color}
\usepackage{multirow}
\usepackage{graphicx}

\usepackage{CJK}
\usepackage{indentfirst}
\usepackage{cases}

\theoremstyle{thmstyleone}
\newtheorem{theorem}{Theorem}

\theoremstyle{thmstyletwo}

\theoremstyle{thmstylethree}

\newtheorem{lemma}{Lemma}

\newtheorem{assumption}{Assumption}

\newtheorem{definition}{Definition}

\begin{document}

\title[Chance constrained nonlinear fractional programming with random benchmark]{Chance constrained nonlinear fractional programming with random benchmark}

\author[1]{\fnm{Tian} \sur{Xia}}\email{xt990221@stu.xjtu.edu.cn}

\author*[1]{\fnm{Jia} \sur{Liu}}\email{jialiu@xjtu.edu.cn}

\affil[1]{\orgdiv{School of Mathematics and Statistics}, \orgname{Xi'an Jiaotong University}, \orgaddress{ \city{Xi'an}, \postcode{710049}, \country{P. R. China}}}




\abstract{
This paper studies the chance constrained fractional programming with a random benchmark. We assume that the random variables on the numerator follow the Gaussian distribution, and the random variables on the denominator and the benchmark follow a joint discrete distribution. Under some mild assumptions, we derive a convex reformulation of chance constrained fractional programming. For practical use, we apply piecewise linear and tangent approximations to the quantile function. We conduct numerical experiments on a main economic application problem.
}


\maketitle


\section{Introduction}
Fractional programming is an important branch of mathematical programming that deals with the optimization of a function defined as the ratio of two functions, i.e., $\frac{f(x)}{g(x)}$, which was initially elaborated by Kantorovich in 1960 \cite{kantorovich1960mathematical}.
Now it has gained significant attention due to its success applications in problems such as information theory \cite{shen2018fractional}, numerical analysis \cite{crouzeix1991algorithms}, management science \cite{zhu2014planning} and finance \cite{goedhart1995financial}. For example, there are many ratios between physical and economical functions, like cost/profit, cost/volume or cost/time.

Among the research of fractional programming, nonlinear fractional programming is a broad subset that deals with the optimization of a function expressed as a ratio of two nonlinear functions \cite{dinkelbach1967nonlinear}. Since widely applicable nature of nonlinearity in fractional functions, it has been applied in many important fields, including financial portfolio optimization \cite{bradley1974fractional}, resource allocation \cite{yang2016energy}, transportation and logistics \cite{schaible2003fractional} and production planning \cite{chu2013integration}. When we face some randomness in a fractional function constraint, we often expect that the constraint can be satisfied in a high probability in many real life problems. Therefore we consider using the chance constraint to guarantee this point, which is especially useful in cases where the violation of
the random constraint might bring severe outcome \cite{kuccukyavuz2022chance}. Furthermore, we study the case where the benchmark of the fractional constraint is random, for example, the expected threshold for the rate of return is random depending on the weather condition for an express company. Gupta \cite{gupta2009chance} studied the chance constrained fractional programming where the benchmark is a decision variable to be maximized in the objective function. Gu et al. \cite{gu2020land} applied chance constrained fractional programming in the area of land use. Belay et al. \cite{belay2022application} study multi-objective chance constrained fractional programming in production planning. Liu et al. \cite{liu2017distributionally} apply the distributionally robust optimization method in fractional programming. Ji et al. \cite{ji2021data} study the data-driven distributionally robust fractional programming based on the Wasserstein ambiguity set. However as far as we know, the chance constrained nonlinear fractional programming with random benchmark, especially its reformulation and solution method, has not been well studied in literature.

In this paper, we consider a type of chance constrained fractional programming with random benchmark. We assume that the random variables on the denominator and benchmark follow a joint discrete distribution, and the random variables on the numerator follows a Gaussian distribution. We derive a convex reformulation of this type of chance constrained fractional programming. Furthermore, we apply piecewise linear and tangent approximations to the quantile function and derive both upper and lower bounds for practical use.

This paper is organized as follows. In Section \ref{CCFPP}, we study the chance constrained fractional programming and derive its convex reformulation under some mild conditions. In Section \ref{ACCFP}, we apply piecewise linear approximations to the reformulated problem. In Section \ref{ne}, we conduct numerical experiments based on a main economic application problem. The last section concludes this paper.














\section{Chance constrained fractional programming problem}\label{CCFPP}

\subsection{Chance constrained fractional programming model}
In this paper, we study the following chance constrained fractional programming (CCFP) problem:
\begin{subequations}\label{1}
\begin{eqnarray}
& \min\limits_{x} & \mathbb{E} \left[a_{0}^{\top}c_{0}(x)\right]\label{1a}\\
& {\rm{s.t.}} & \mathbb{P}\left(\frac{a_1^{\top}c(x)+b_1}{a_2^{\top}c(x)+b_2}\le\gamma\right)\ge1-\epsilon,\label{1b}\\
&& x\in\mathcal{X},\label{1c}
\end{eqnarray}
\end{subequations} where $x\in \mathbb{R}^m$ is the decision variable,
$$c_{0}(x)=[c_{0,1}(x),c_{0,2}(x),...,c_{0,n}(x)]^{\top}\in\mathbb{R}^{n_0},\ c(x)=[c_{1}(x), c_{2}(x),...,c_{n}(x)]^{\top}\in \mathbb{R}^{n}, $$
$\mathcal{X}$ is a closed convex set and $a_{0}\in\mathbb{R}^{n_0}, a_{1}=[a_{1,1},a_{1,2},...,a_{1,n}]^{\top}, a_{2}=[a_{2,1},a_{2,2},...,a_{2,n}]^{\top}\in \mathbb{R}^n, b_{1}, b_{2}\in\mathbb{R}$,
$\gamma\in \mathbb{R}$ are random variables, $\epsilon\in (0,1)$ is the confidence level of the chance constraint. Note that the denominator $a_2^{\top}c(x)+b_2\neq 0$ always, and we hold the following assumptions in this paper.

\begin{assumption}\label{2}
For the CCFP problem \eqref{1}, we assume that $a_{1}, b_{1}
$ 
follow the Gaussian distribution, $a_{2}, b_{2}, \gamma$ follows a discrete distribution with $J$ realizations. $c(x)\ge0$ for any $x\in\mathcal{X}$.
 $\{a_{1}, b_{1}\}$ are independent to $ \{a_{2}, b_2,\gamma\}$.
The mean vector of $a_{0}^{\top}$ is $\mu_{0}^{\top}=(\mu_{0,1}, \mu_{0,2},..., \mu_{0,n_0})$, the mean vector of $a_{1}^{\top}$ is $\mu_{1}^{\top}=(\mu_{1,1}, \mu_{1,2},..., \mu_{1,n}),$ and the mean value of $b_1$ is $l_1$. 
For $a_{2}, b_{2}, \gamma$, there exist $a_{2}^{j}\ge 0, b_{2}^{j}>0, r_{j}
, j=1,...,J,$ such that $$\mathbb{P}\left(a_{2}=a_{2}^{j},b_{2}=b_{2}^{j}, \gamma=r_{j} \right)=p_{j}$$ and $\sum\limits_{j=1}^{J}p_j=1$.
The covariance matrix of vector 
$[a_{1,1},a_{1,2},..., a_{1,n}, b_{1}]^{\top}$ is $\Gamma=\{\sigma_{p,q}, p,q=1,...,n+1\}$, and $\sigma_{p,q}\ge 0, p,q=1,...,n+1.$
\end{assumption}

Given Assumption \ref{2}, we could get the following reformulation of \eqref{1}.

\begin{theorem}
Given Assumption \ref{2}, the CCFP problem \eqref{1} has the following reformulation:
\begin{subequations}\label{3}
\begin{eqnarray}
& \max\limits_{x,z} & \mu_{0}^{\top}c_{0}(x)\label{2a}\\
& {\rm{s.t.}} & \Phi^{-1}(z_{j})\sqrt{\sum\limits_{p=1}^{n}\sum\limits_{q=1}^{n}{\sigma}_{p,q}c_{p}(x)c_{q}(x)+\sum\limits_{p=1}^{n}{\sigma}_{p,n+1}c_{p}(x)+\sum\limits_{q=1}^{n}{\sigma}_{n+1,q}c_{q}(x)+\sigma_{n+1,n+1}}\quad\quad \label{2b}\\
&&\quad +(\mu_1-r_{j}a_2^{j})^{\top}c(x)\le r_{j}b_{2}^{j}-l_{1},j=1,...,J,\nonumber\\ 
&& \sum\limits_{j=1}^{J} z_{j}p_{j}\ge 1-\epsilon, \label{2c}\\
&& 0\le z_j \le 1, j=1,...,J,\\
&& x\in\mathcal{X},\label{2f}
\end{eqnarray}
\end{subequations} where $\Phi^{-1}(\cdot)$ is the quantile function of a standard Gaussian distribution $N(0,1)$.
\end{theorem}

\begin{proof}
Given Assumption \ref{2}, it is obvious that $\eqref{1a}\iff\eqref{2a}$. The constraint \eqref{1b} has the following equivalent reformulation.
\begin{equation}\label{pl}
\eqref{1b} \Longleftrightarrow \sum\limits_{j=1}^{J}\mathbb{P}(a_{2}=a_{2}^{j}, b_{2}=b_{2}^{j}, \gamma=r_j)\mathbb{P}\left(\frac{a_1^{\top}c(x)+b_1}{{a_2^{j}}^{\top}c(x)+b_2^{j}}\le r_j\right)\ge1-\epsilon.
\end{equation} It is because the random variables $a_{2}, b_{2}$ and $\gamma$ follow a discrete distribution and $\{a_{1}, b_{1}\}$ are independent to $ \{a_{2}, b_2,\gamma\}$ as stated in Assumption \ref{2}. By introducing the auxiliary variable $z_j,j=1,...,J$, we have
\begin{subnumcases}
    {\eqref{pl}\Longleftrightarrow}
    \mathbb{P}\left(\frac{a_1^{\top}c(x)+b_1}{{a_2^{j}}^{\top}c(x)+b_2^{j}}\le r_j\right)\ge z_j, j=1,...,J, \label{pll}\\
\sum\limits_{j=1}^{J}p_{j}z_{j}\ge 1-\epsilon,\\
0\le z_j \le 1, j=1,...,J.
\end{subnumcases}
As $a_{2}^{j}\ge 0, c(x)\ge0, b_{2}^{j}> 0$ as defined in Assumption \ref{2}, we have
\begin{equation}\label{plll}
\eqref{pll}\Longleftrightarrow\mathbb{P}\left[(a_1-r_{j}a_2^{j})^{\top}c(x)+b_1-r_{j}b_{2}^{j}\le 0\right]\ge z_j, j=1,...,J.
\end{equation}



Let $y=(c_{1}(x),c_{2}(x),...c_{n}(x),1),$ $h^{j}=(a_{1,1}-r_{j}a_{2,1}^{j},a_{1,2}-r_{j}a_{2,2}^{j},..., a_{1,n}-r_{j}a_{2,n}^{j}, b_{1}-r_{j}b_{2}^{j})$. 
Then $(a_1-r_{j}a_2^{j})^{\top}c(x)+b_1-r_{j}b_{2}^{j}={h^{j}}^{\top}y$. By \cite{patel1996handbook}, each element of $h^{j}$ follows the Gaussian distribution. Therefore we know from \cite{prekopa2003probabilistic} that
\begin{equation}\label{fgy}
    \eqref{plll}\iff\mathbb{P}({h^{j}}^{\top}y\le 0)\ge z_{j}\iff \mathbb{E}[h^{j}]^{\top}y+\Phi^{-1}(z_{j})\sqrt{{\rm{Var}}({h^{j}}^{\top}y)}\le 0,, j=1,...,J.
\end{equation}
By taking the expressions of $\mathbb{E}[h^{j}]$ and ${\rm{Var}}({h^j}^{\top}y)$ in \eqref{fgy}, we have the reformulation of $(4)$,
\begin{equation}\label{plllll}
\begin{aligned}
(4)\Longleftrightarrow \left\{\begin{array}{l}
\Phi^{-1}(z_{j})\sqrt{\sum\limits_{p=1}^{n}\sum\limits_{q=1}^{n}{\sigma}_{p,q}c_{p}(x)c_{q}(x)+\sum\limits_{p=1}^{n}{\sigma}_{p,n+1}c_{p}(x)+\sum\limits_{q=1}^{n}{\sigma}_{n+1,q}c_{q}(x)+\sigma_{n+1,n+1}}\\
\quad +(\mu_1-r_{j}a_{2}^{j})^{\top}c(x)\le r_{j}b_{2}^{j}-l_{1},\  j=1,...,J,\\ 
\sum\limits_{j=1}^{J} p_{j}z_{j}\ge 1-\epsilon,\\
0\le z_j \le 1, j=1,...,J.
\end{array}\right.
\end{aligned}
\end{equation}
By taking $\eqref{plllll}$ back to \eqref{1b}, we complete the proof.
\end{proof}




\subsection{Convex reformulation of chance constrained fractional programming}
In this section, we propose a condition under which \eqref{3} can be transformed into a convex programming problem. Firstly, we introduce the following definition and lemmas.

\begin{definition}
    We say a positive-valued function $f(x)$ is logarithmic convex (log-convex) w.r.t. $x$ if $\log{(f(x))}$ is convex w.r.t. $x$.
\end{definition}

\begin{lemma}[\cite{liu2022chance}, Lemma 2.1]\label{i}
    $\log{(\Phi^{-1}(x))}$ is monotonically on $(0.5,1)$, and convex on $[\Phi(1),1)$.
\end{lemma}



\begin{lemma}\label{r}
Given $\epsilon\le \min\limits_{1\le j\le J}p_{j}(1-\Phi(1))$, any feasible solution of \eqref{3} satisfies that $z_{j}\ge \Phi(1)$. 
\end{lemma}

\begin{proof}
Given $\epsilon\le\min\limits_{1\le j\le J}p_{j}(1-\Phi(1))$, by \eqref{2c}, for any $j=1,...,J$, we have
$$z_{j} \ge \frac{1}{p_j}\left(1-\epsilon-\sum\limits_{i=1,..,J,i\neq j}z_{i}p_{i}\right)\ge \frac{1}{p_j}\left(1-\epsilon-\sum\limits_{i=1,..,J,i\neq j}p_i\right)=1-\frac{\epsilon}{p_j}\ge \Phi(1).$$
\end{proof}

\begin{lemma}[\cite{liu2022chance}, Lemma 2.3]\label{8}
Given two nonnegative valued, convex and twice differentiable functions $f(x):\mathbb{R}^n\rightarrow \mathbb{R}_{+}$, $g(y):\mathbb{R}^n\rightarrow \mathbb{R}_{++}$. $f(x)g(y)$ is convex if and only if 
$$f(x)\nabla^{2}g(y)-\frac{1}{g(y)}\nabla^{2}f(x)^{-1}\nabla f(x)\nabla g(y)^{\top}\nabla f(x)\nabla g(y)^{\top}\succeq_{s}0,$$
where $A\succeq_{s}0$ means that $A$ is a positive semidefinite matrix.
\end{lemma}


\begin{lemma}[\cite{liu2022chance}, Lemma 2.4]\label{9}
If $f(x)\nabla^{2}f(x)-\nabla f(x)\nabla f(x)^{\top}\succeq_{s}0$ and $g(y)\nabla^{2}g(y)-\nabla g(y)\nabla g(y)^{\top}\succeq_{s}0$, then $f(x)\nabla^{2}g(y)-\frac{1}{g(y)}\nabla^{2}f(x)^{-1}\nabla f(x)\nabla g(y)^{\top}\nabla f(x)\nabla g(y)^{\top}\succeq_{s}0$.
\end{lemma}

\begin{lemma}[\cite{boyd2004convex}, Section 3.5.2]\label{10}
For positive valued and twice differentiable function $f(x)$, $f(x)\nabla^{2}f(x)-\nabla f(x)\nabla f(x)^{\top}\succeq_{s}0$ if and only if $f(x)$ is log-convex.
\end{lemma}

\begin{assumption}\label{4}
For the optimization problem \eqref{3}, we assume that $\mu_{1,i}-r_{j}a_{2,i}^{j}\ge 0$, $i=1,...,n$, $j=1,2,...,J$, $\epsilon\le \min\limits_{1\le j\le J}p_{j}(1-\Phi(1))$, $c_{0}(x)$ is convex, and $c_{i}(x)$ is non-negative valued and log-convex w.r.t. $x$ for $i=1,2,...,n.$
\end{assumption}

With the lemmas and assumption above, we have the following theorem which transforms \eqref{3} into a convex reformulation.

\begin{theorem}\label{con}
Given Assumptions \ref{2},\ref{4}, the CCFP problem \eqref{3} is a convex optimization problem and has the following reformulation:
\begin{subequations}\label{CCF}
\begin{eqnarray}
& \max\limits_{x,z,Y} & \mu_{0}^{\top}c_{0}(x)\\
& {\rm{s.t.}} & (\mu_1-r_{j}a_2^{j})^{\top}c(x)+\Vert Y^{j}\Vert_{F}\le r_{j}b_{2}^{j}-l_{1}, j=1,...,J,\label{6b}\\
&& \sqrt{ {\sigma}_{p,n+1}}\exp{\left\{\frac{1}{2}\log{c_p(x)}+\log{\Phi^{-1}(z_j)}\right\}}\le Y_{p,n+1}^{j}, p=1,...,n, j=1,...,J,\label{6c}\\
&& \sqrt{ {\sigma}_{n+1,q}}\exp{\left\{\frac{1}{2}\log{c_q(x)}+\log{\Phi^{-1}(z_j)}\right\}}\le Y_{n+1,q}^{j}, q=1,...,n, j=1,...,J,\label{6d}\\
&& \sqrt{ {\sigma}_{n+1,n+1}}\exp{\left\{\log{\Phi^{-1}(z_j)}\right\}}\le Y_{n+1,n+1}^{j}, j=1,...,J,\\
&& \sqrt{ {\sigma}_{p,q}}\exp{\left\{\frac{1}{2}\log{c_p(x)}+\frac{1}{2}\log{c_q(x)}+\log{\Phi^{-1}(z_j)}\right\}}\le Y_{p,q}^{j}, p,q=1,...,n, j=1,...,J,\quad \quad \label{10f}\\
&& \sum\limits_{j=1}^{J} z_{j}p_{j}\ge 1-\epsilon, \label{10g}\\
&& 0\le z_j \le 1, j=1,...,J,\\
&& x\in\mathcal{X}\label{10i}.
\end{eqnarray}
\end{subequations}
\end{theorem}

\begin{proof}
    By Assumption \ref{4}, $\epsilon\le \min\limits_{1\le j\le J}p_{j}(1-\Phi(1))$, so we have $z_{j}\ge \Phi(1)>0$ by Lemma \ref{r}. By Lemma \ref{i}, we know that $\Phi^{-1}(z_{j})$ is log-convex. As each element of $y=(c_1(x), c_2(x)$,$...$, $c_n(x), 1)$ is log-convex, $\sqrt{y_i}$ is log-convex, then by Lemmas \ref{8}-\ref{10}, we have $y_{p}^{\frac{1}{2}}y_{q}^{\frac{1}{2}}\Phi^{-1}(z_{j})$ is a convex function for any $p,q,j$. Therefore given ${\sigma}_{p,q}\ge 0$,
    $$\Phi^{-1}(z_j)\sqrt{\sum\limits_{p=1}^{n+1}\sum\limits_{q=1}^{n+1}{\sigma}_{p,q}y_{p}y_{q}}=\sqrt{\sum\limits_{p=1}^{n+1}\sum\limits_{q=1}^{n+1}\left(\sqrt{{\sigma}_{p,q}}y_{p}^{\frac{1}{2}}y_{q}^{\frac{1}{2}}\Phi^{-1}(z_{j})\right)^2}$$ is a convex function, as it is a composition of an outer-level $L_2$-norm function and the inner-level convex terms. Then \eqref{2b} is equivalent to
\begin{equation}\label{7}
\left\{\begin{array}{l}
(\mu_1-r_{j}a_2^{j})^{\top}c(x)+\Vert Y^{j}\Vert_{F}\le r_{j}b_{2}^{j}-l_{1}, j=1,...,J,\\
\sqrt{ {\sigma}_{p,q}}\exp{\left\{\frac{1}{2}\log{y_p}+\frac{1}{2}\log{y_q}+\log{\Phi^{-1}(z_j)}\right\}}\le Y_{p,q}^{j}, p,q=1,...,n+1, j=1,...,J.
\end{array}\right. 
\end{equation}
Taking \eqref{7} and the expression of $y$ back to \eqref{3}, we get the reformulation and finish the proof.
\end{proof}

\subsection{the case when c(x) is linear}
 new reformulation

 log-transform from log(x) to t.

\section{Approximation of chance constrained fractional programming}\label{ACCFP}
As $\Phi^{-1}(\cdot)$ is nonelementary, we need to approximate the function $\Phi^{-1}(\cdot)$ to more efficiently solve \eqref{CCF}. In this section, we derive two convex approximations of the CCFP problem \eqref{CCF} using the piecewise linear and tangent approximation technique, respectively.

\subsection{Piecewise linear approximation}\label{331}
We apply the piecewise linear approximation technique \cite{cheng2012second} to $\log\Phi^{-1}(\cdot)$. Without loss of generality, we choose $K+1$ points $\xi_{k}, k=1,...,K+1$ in the interval $[\Phi(1), 1]$ such that $\xi_1<\xi_2<...<\xi_{K+1}$, where $\xi_{1}=\Phi(1)$, $\xi_{K+1}$ is close enough to $1$. We use the piecewise linear function $F(z)=\max\limits_{k=1,...,K}F_{k}(z)$ to approximate $\log\Phi^{-1}(z)$, where $$F_{k}(z):=\log\Phi^{-1}(\xi_k)+\frac{z-\xi_k}{\xi_{k+1}-\xi_{k}}(\log\Phi^{-1}(\xi_{k+1})-\log\Phi^{-1}(\xi_k))=u_{k}z+t_{k}, k=1,...,K,$$ and


\begin{equation}\label{kop}
u_{k}=\frac{\log\Phi^{-1}(\xi_{k+1})-\log\Phi^{-1}(\xi_k)}{\xi_{k+1}-\xi_k},\ \ t_{k}=\frac{\xi_{k+1}\log\Phi^{-1}(\xi_k)-\xi_{k}\log\Phi^{-1}(\xi_{k+1})}{\xi_{k+1}-\xi_k}.
\end{equation} Then we replace $\log\Phi^{-1}(\cdot)$ in \eqref{CCF} by $F(\cdot)$, and get the following upper convex approximation of CCFP problem \eqref{CCF}.

\begin{theorem}\label{jkl}
    Using the piecewise linear function $F(\cdot)$, we have the following convex approximation of the CCFP problem \eqref{CCF}:
\begin{subequations}\label{TY}
\begin{eqnarray}
& \max\limits_{x,z,Y} & \mu_{0}^{\top}c_{0}(x)\\
& {\rm{s.t.}} & \sqrt{ {\sigma}_{p,n+1}}\exp{\left\{\frac{1}{2}\log{c_p(x)}+u_{k}z_{j}+t_{k}\right\}}\le Y_{p,n+1}^{j}, p=1,...,n, j=1,...,J, k=1,...,K,\label{okl}\\
&& \sqrt{ {\sigma}_{n+1,q}}\exp{\left\{\frac{1}{2}\log{c_q(x)}+u_{k}z_{j}+t_{k}\right\}}\le Y_{n+1,q}^{j}, q=1,...,n, j=1,...,J, k=1,...,K,\label{11c}\\
&& \sqrt{ {\sigma}_{n+1,n+1}}\exp{\left\{u_{k}z_{j}+t_{k}\right\}}\le Y_{n+1,n+1}^{j}, j=1,...,J, k=1,...,K,\\
&& \sqrt{ {\sigma}_{p,q}}\exp{\left\{\frac{1}{2}\log{c_p(x)}+\frac{1}{2}\log{c_q(x)}+u_{k}z_{j}+t_{k}\right\}}\le Y_{p,q}^{j}, p,q=1,...,n, j=1,...,J,\quad \quad\quad \label{17f}\\
&& \quad \quad k=1,...,K,\nonumber\\
&& \eqref{6b}, \eqref{10g}-\eqref{10i}.
\end{eqnarray}
\end{subequations} The optimal value of \eqref{TY} is a upper bound of problem \eqref{CCF}. Moreover, if $\lim\limits_{K\rightarrow+\infty}\max\limits_{k=1,...,K}|\xi_{k+1}-\xi_{k}|\rightarrow 0$, the optimal value of \eqref{TY} converges to that of problem \eqref{CCF}.
\end{theorem}

\begin{proof}
        By Assumption \ref{4} and Lemma \ref{r}, $z_{j}\in[\Phi(1),1]$. Then by Lemma \ref{i}, we have $\log\Phi^{-1}(z_j)$ is convex w.r.t. $z_j$ in $[\Phi(1),1]$. Therefore for any two points $\xi_{k}, \xi_{k+1}, k=1,...,K,$ we have $$\log\Phi^{-1}(z_j)\le u_{k}z_{j}+t_{k},\ \ \forall z_{j}\in[\xi_{k},\xi_{k+1}],$$ where $u_{k}$ and $t_{k}$ are defined in \eqref{kop}. Thus we have,
        \begin{equation}\label{dfgg}
            \log\Phi^{-1}(z_{j})\le\max\limits_{k=1,...,K}u_{k}z_{j}+t_{k},\ \ \forall z_{j}\in[\Phi(1), 1].
        \end{equation}
        And for constraints \eqref{6c} and \eqref{okl}, we have
    \begin{equation}
    \begin{split}
        \left\{x,z: \sqrt{ {\sigma}_{p,n+1}}\exp{\left\{\frac{1}{2}\log{c_p(x)}+u_{k}z_{j}+t_{k}\right\}}\le Y_{p,n+1}^{j}, p=1,...,n, j=1,...,J, k=1,...,K\right\}\\\subseteq\left\{x,z: \sqrt{ {\sigma}_{p,n+1}}\exp{\left\{\frac{1}{2}\log{c_p(x)}+\log{\Phi^{-1}(z_j)}\right\}}\le Y_{p,n+1}^{j}, p=1,...,n, j=1,...,J\right\}.
    \end{split}
    \end{equation} Similar relationship holds for constraints \eqref{6d}-\eqref{10f} and \eqref{11c}-\eqref{17f}, we have that the approximation problem \eqref{TY} provides a upper bound of the problem \eqref{CCF}.

Moreover, 
when $K\rightarrow+\infty$, the gap between adjacent points, i.e. $\max\limits_{k=1,...,K}|\xi_{k+1}-\xi_{k}|$, trends to zero,
the inequality \eqref{dfgg} becomes tight. Furthermore, by the convexity of $\eqref{6c}-\eqref{10f}$ and $\eqref{okl}-\eqref{17f}$, we know that the distance between the sets constrained by $\eqref{6c}-\eqref{10f}$ and $\eqref{okl}-\eqref{17f}$ is small enough when $K\rightarrow+\infty$. As the objective functions and other constraints in \eqref{CCF} and \eqref{TY} are the same, the optimal value of \eqref{TY} converges to that of problem \eqref{CCF} when $K\rightarrow+\infty$.
\end{proof}
















\subsection{Piecewise tangent approximation}\label{ikm}

 We apply the piecewise tangent approximation technique \cite{liu2016stochastic} to $\log\Phi^{-1}(\cdot)$ to derive a lower bound. We choose $K$ points $\xi_{k}, k=1,...,K$ from the interval $[\Phi(1),1]$. 
 Then we use the piecewise tangent function $F(z)=\max\limits_{k=1,...,K}F_{k}(z)$ to approximate $\log\Phi^{-1}(z)$ in problem \eqref{CCF}, where $$F_{k}(z)=d_{k}z+b_{k},$$ and 
\begin{equation}\label{appr}
d_{k}=\frac{(\Phi^{-1})^{(1)}(\xi_k)}{\Phi^{-1}(\xi_{k})},\ \ \ b_{k}=-d_{k}\xi_{k}+\log\Phi^{-1}(\xi_{k}),\ k=1,...,K,
\end{equation} such that 
\begin{equation}\label{12e}
    F_{k}(z)\le \log\Phi^{-1}(z), \ \ \forall z\in[\Phi(1),1],\  k=1,...,K,\  j=1,...,J.
\end{equation} Here the tangent approximation $F_{k}(z)=d_{k}z+b_{k}$ are chosen from the tangent lines of $\log\Phi^{-1}(z)$ at different points in $[\Phi(1),1]$. 
Similar to Theorem \ref{jkl}, when $\lim\limits_{K\rightarrow+\infty}\max\limits_{k=1,...,K}|\xi_{k+1}-\xi_{k}|\rightarrow 0$, the convex bound \eqref{12e} is asymptotically tight. By replacing all $u_{k}z_{j}+t_{k}$ with $d_{k}z_{j}+b_{k}$ in \eqref{TY},  
we get a lower convex approximation of CCFP problem \eqref{CCF}.

\section{Numerical experiments}\label{ne}
In this section, we carry out numerical experiments on a typical nonlinear fractional programming: the Main Economic Application problem \cite{schaible1983fractional}. We use FMINCON in MATLAB to solve the approximated optimization problems. All experiments are conducted on a PC with AMD Ryzen 7 5800H CPU and 16.0 GB RAM.

Suppose a certain company manufacture $n=5$ different products. For $i=1,...,n$, let $a_{1,i}$ be the cost for the manufacturing of one unit of product $i$, $a_{2,i}$ be the profit gained by the company from a unit of product $i$, $a_{3,i}$ be the expenditure quota of manufacturing a unit of product $i$, $\alpha, \beta$ be the upper and lower bounds for the expenditure quota such that $\alpha\le\mathbb{E}(a_{3}^{\top}x)\le\beta$, $b_{1}, b_{2}$ be some constant cost and profit of manufacturing, respectively, and $x_{i}$ be the production volume of product $i$. For the operation of a company, it aims at maximizing its expected total profit $\mathbb{E}(a_2^{\top}x)$ and keeping the (total cost)/(total profit) ratio not exceeding a reference ratio 
as much as possible. 
We apply a chance constraint to describe the requirement on the ratio. The cost, profit and expenditure quota of manufacturing are dependent of external factors. Here we consider weather as our stochastic influencing factor. In particular, we assume that the weather is classified into sunny days and rainy days. On sunny days, the reference (cost)/(profit) ratio of the company is $r_{1}=0.4$. On rainy days, the reference ratio is $r_{2}=0.6$. We assume that the probability of sunny days is $p_1=0.7$ and the probability of rainy days is $p_2=0.3$. In our experiment, we set the confidence level $\epsilon=0.02<\min\limits_{1\le j\le 2}p_{j}(1-\Phi(1))$, $\mu_{1}=[40, 80,60,70,70]$, $a_{2}^{1}=[55, 100, 80, 95, 90]$, $a_{2}^{2}=[45, 90, 70, 85, 80]$, $b_{2}^{1}=5$, $b_{2}^{2}=3$, $\mu_{3}=[30, 50, 20, 40, 30]$, $l_{1}=2$, $[\alpha, \beta]=[50,100]$, and
$$ {\Gamma}=\begin{pmatrix}
 6 & 7 & 1 & 2 & 2 & 1 \\
 7 & 9 & 2 & 4 & 4 & 2 \\
 1 & 2 & 6 & 1 & 1 & 1 \\
 2 & 4 & 1 & 9 & 3 & 1 \\
 2 & 4 & 1 & 3 & 7 & 1 \\
 1 & 2 & 1 & 1 & 1 & 6 
 \end{pmatrix}.$$ Here $ {\Gamma}$ is symmetric positive definite and all settings are in line with Assumptions \ref{2},\ref{4}.
 We choose the points $\xi_{k}, k=1,2,...,K$ uniformly in $[\Phi(1),1]$.
 We derive numerical results of two approximation models as shown in Table 1.


 \begin{table}[h!]
 \caption{\normalsize Numerical results of two approximation models.}
\centering
\begin{tabular}{c|c|c|c|c|c}
\hline
$K$ & Lower bound & CPU(s) & Upper bound & CPU(s) & Gap\\ \hline
$3$ & 292.3779 & 9.4752 & 292.5845 & 7.2409 & 0.2066 \\ \hline
$4$ & 292.5457 & 10.3780 & 292.6720 & 9.9378 & 0.1263 \\ \hline
$5$ & 292.5142 & 11.0373 & 292.5924 & 8.4648 & 0.0782 \\ \hline
$6$ & 292.4880 & 11.3402 & 292.5377 & 11.3516 & 0.0497 \\ \hline
\end{tabular}
\end{table}

From Table 1, we see that as the number of segments $K$ increases, the gap between the lower bound and the upper bound decreases, which presents the tendency of convergence and it is consistent with the conclusions in Section \ref{331} and Section \ref{ikm}. Moreover, the CPU time does not increase violently with $K$ increasing, which means that our approximation approaches are tractable and efficient.

\section*{Conclusion}
In this paper, we study the chance constrained nonlinear fractional programming where some random variables follow a discrete distribution and others follow Gaussian distribution. Under some mild assumptions, we derive its deterministic convex reformulation. For efficient computation, we apply the piecewise linear and tangent approximation methods to derive upper and lower approximations. Numerical experiments on the main economic application problem validate the efficiency and tractability of approximation models. Considering the distributionally robust chance constrained fractional programming is a promising topic for further study.

\section*{Acknowledgements}

This research was supported by National Key R\&D Program of China under No. 2022YFA1004000 and National Natural Science Foundation of China under No. 11991023 and 11901449.


\end{document}